\documentclass{amsart} 
\usepackage[utf8]{inputenc}
\usepackage{amsthm}
\usepackage{amsmath}

\usepackage{enumerate}   
\usepackage{amssymb}
\usepackage{fancyhdr} 
\usepackage{xcolor}
\newcommand{\norm}[1]{\left\lVert#1\right\rVert}
\usepackage{graphicx}
\usepackage{ esint }
\newtheorem{theorem}{Theorem}[section]
\newtheorem{proposition}[theorem]{Proposition}

\newtheorem{lemma}[theorem]{Lemma}
\theoremstyle{definition}

\theoremstyle{remark}

\theoremstyle{remark}
\newtheorem*{Littlewood}{Littlewood's subordination principle}
\theoremstyle{remark}
\newtheorem*{Koebe}{Koebe's growth and distortion theorem}
\renewcommand\Re{\operatorname{Re}}
\renewcommand\Im{\operatorname{Im}}

\usepackage{url}
\usepackage[hidelinks]{hyperref}
\hypersetup{
	colorlinks=true,
	linkcolor=cyan,
	filecolor=cyan,
	citecolor =cyan,      
	urlcolor=cyan,
}
\author{Alexandru Aleman}
\address{Lund University, Mathematics, Faculty of Science, P.O. Box 118, S-221 00 Lund, Sweden}
\email{alexandru.aleman@math.lu.se}
\author{Athanasios Kouroupis}
\address{Department of mathematics, KU Leuven, Celestijnenlaan 200B, 3001, Leuven, Belgium}
\email{athanasios.kouroupis@kuleuven.be}

\title{Brennan's conjecture holds for semigroups of holomorphic functions}
\date{}

\begin{document}

\begin{abstract}
In the present note, we give a short proof of Brennan's conjecture in the special case of continuous semigroups of holomorphic functions. We apply classical techniques of complex analysis in conjunction with recent results on B\'{e}koll\'{e}-Bonami weights and spectra of integration operators.
\end{abstract}
\maketitle
\section{Introduction}
J. E. Brennan in \cite{BRE78} conjectured that if $f:\mathbb{D}\mapsto \mathbb{C}$ is a conformal map. Then, for every $p\in(-2,\frac{2}{3})$ the $p$-integral mean of the derivative is finite:
\begin{equation}\label{eq:brennan}
\int\limits_{\mathbb{D}}\left|f'(w)\right|^p\,dA(w)<\infty,
\end{equation}
where $dA(z)=\frac{dx\,dxy}{\pi}$ is the normalized Lebesgue measure of the unit disk $\mathbb{D}$.

Brennan's conjecture is one of the most famous remaining open problems in the field of geometric function theory. The best result so far is due to Alan Sola \cite{alan2, alan}, where he proved that the conjecture holds for the values $p\in(-p_0,\frac{2}{3}),\,p_0\sim 1.76$. We refer the interested reader also to the article of Hedenmalm and Shimorin \cite{HS05}. For further examples of univalent functions satisfying the conjecture and connections with complex dynamics, we refer to \cite{Volberg98}.      

In the present note, we will restrict ourselves to conformal maps $f$ that can be embedded into a continuous semigroup of holomorphic functions in the unit disk.
\begin{theorem}\label{main}
Let $\{\phi_t(z)\}_{t\geq 0}$ be a continuous semigroup of conformal self-maps of the unit disk. Then, for every $p\in(-2,\frac{2}{3})$
\begin{equation}\label{eq:sembrennan}
\int\limits_{\mathbb{D}}\left|\phi_t'(w)\right|^p\,dA(w)<\infty.
\end{equation}
\end{theorem} 
A (continuous) semigroup of holomorphic functions, \cite{BCDM20}, is a family $\{\phi_t(z)\}_{t\geq 0}$ of conformal self-maps of the unit disk, such that:
\begin{enumerate}[(i)]
	\item $\phi_0(z)=z,\,z\in\mathbb{D}$.
	\item $\phi_{t+s}(z)=\phi_t\circ\phi_s(z),\,z\in\mathbb{D}$.
	\item For every $z\in\mathbb{D}$, the map $t\mapsto \phi_t(z)$ is continuous for $t\geq0$.
\end{enumerate}
%The result is optimal in the sense that the non-trivial iterates of the semigroup $\phi_t(z)= k^{-1}\left(k(z)+t\right)$, where $k(z)=\frac{z}{(1-z)^2}$ is the Koebe function, satisfy \eqref{eq:sembrennan} only for  $p\in(-2,\frac{2}{3})$.
\subsection*{Motivation}  Bertilsson, see \cite{BER99,HH2002}, proved that Brennan's conjecture is equivalent to the following estimate:
\[\int_0^{2\pi}|f'(re^{i\theta})|^{-2}\,d\theta\leq C\frac{|f'(0)|^{-2}}{1-r},\qquad f\text{ conformal},\qquad r\in(0,1),\]
where $C$ is an absolute constant. By a breakthrough result of N. Makarov the validity of Brennan's conjecture for bounded maps would imply the conjecture in its full generality, see \cite{carleson1994some,makarov1999fine}. Thus, for a potential proof of the conjecture it is sufficient to work only with slit-maps or their corresponding Löwner chains, \cite{GOL69}. But, what if we give a Löwner chain an extra additive structure?

This was exactly the reason why we are considering semigroups of holomorphic functions.
Note that a semigroup with Denjoy-Wolff point at the origin behaves like a Löwner chain, meaning that it  satisfies the differential equation
\[\frac{d\phi_t(z)}{dt}=-\phi_t(z)p(\phi_t(z)),\qquad\Re p(z)\geq 0,\]
and the following subordination relation
\[\phi_s\prec\phi_t,\qquad 0\leq t\leq s.\]

 The theory of semigroups in the unit disk has gained popularity the last years, \cite{betsakos2024monotonicity, bracci2023complete, contreras2024slope,contreras2023centralizers, contreras2023uniform, gumenyuk2024angular, kourou2024rates,zarvalis2023quasi}. It will be interesting to see if one can work towards Brennan's conjecture having as a starting point the result of this note and establishing semigroup-fication (\cite{bracci2020semigroup}) or approximation theorems between conformal functions and semigroups.

\subsection*{Acknowledgments}
The authors thank Manuel Contreras, Alan Sola and Georgios Stylogiannis for their helpful comments and discussions.

A. Kouroupis received support from Research Foundation--Flanders (FWO), Odysseus Grant No. G0DDD23N

\subsection*{Notation}
Throughout the paper, if $f,g: E\to \mathbb R$ are two functions defined on the same set $E$, the notation $f\lesssim  g$ will mean that there exists some constant $C > 0$ such that $f \leq Cg$ on $E$. We will write  $f \sim g$ when $f\lesssim  g$ and $g\lesssim  f$.
\section{Background Material}
\subsection{Holomorphic Semigroups and conformal functions}
We will make a short presentation on the topic and we refer the interested reader to \cite{BCDM20} for a thorough introduction and an overview of recent developments.

We recall that a holomorphic self-map of the unit disk $\phi$, which is not the identity, satisfies one of the following:
\begin{enumerate}[(i)]
\item $\phi$ has a unique fixed point $\tau$ in $\mathbb{D}$.\\
\item $\phi$ has no fixed point in $\mathbb{D}$ and there exists a unique point $\tau$ in the unit circle $\mathbb{T}$ such that
\[\angle\lim\limits_{z\rightarrow\tau}\phi(z)=\tau\qquad \text{and} \qquad \phi'(\tau):=\angle\lim\limits_{z\rightarrow\tau}\phi'(z)=\liminf_{z\rightarrow\tau}\frac{1-|\phi(z)|}{1-|z|}\leq 1.\]
\end{enumerate}
The point $\tau\in\overline{\mathbb{D}}$ is called the Denjoy-Wolff point of $\phi$. If $\tau\in\mathbb{D}$ then the map $\phi$ is called elliptic, if $\tau\in\mathbb{T}$ and $\phi'(\tau)=1$ it is called parabolic and finally if $\tau\in\mathbb{T}$ and $\phi'(\tau)<1$ it is called hyperbolic. 

If $\{\phi_t(z)\}_{t\geq 0}$ is a continuous non-trivial semigroup in $\mathbb{D}$ then all iterates different from the identity have the same Denjoy-Wolff point $\tau\in\overline{\mathbb{D}}$. Furthermore, if $\tau\in\mathbb{T}$ then
\[\phi_t'(\tau)=e^{-\lambda t},\qquad \Re \lambda\geq 0.\]
The semigroup $\{\phi_t(z)\}_{t\geq 0}$ is called elliptic, parabolic or hyperbolic if $\phi_1$ is elliptic, parabolic or hyperbolic, respectively. From now on we will assume that every semigroup is non-trivial and continuous.

For every semigroup $\{\phi_t(z)\}_{t\geq 0}$ of the unit disk there exists a conformal map $h(z)$ in the unit disk $\mathbb{D}$, the Koenings map, such that:
\begin{equation}
\phi_t(z)=h^{-1}\circ\psi_t\circ h(z),\qquad z\in \mathbb{D},
\end{equation}
where
\begin{enumerate}[(i)]
\item $\psi_t(z)= e^{-\lambda t}z,\, \Re\lambda\geq0$, when $\{\phi_t(z)\}_{t\geq 0}$ is an elliptic semigroup.\label{elliptic}
\item $\psi_t(z)= z+it,$ when $\{\phi_t(z)\}_{t\geq 0}$ is a non-elliptic semigroup (parabolic or hyperbolic).
\end{enumerate}
In the first case, the Koenig map $h$ is a $\lambda$-spirallike function with respect to $\tau$, that is 
\begin{equation}\label{koenigel}
\Re \left(\frac{1}{\lambda}(z-\tau)(1-\overline{\tau}z)\frac{h'(z)}{h(z)}\right)\geq 0.
\end{equation}

An equivalent geometric definition of a $\lambda$-spirallike function $h$ with respect to $\tau\in\mathbb{D}$ is the following:
\begin{enumerate}[(i)]
	\item The function $h$ is conformal in the unit disk and $h(\tau)=0$.
	\item If a point $z$ exists in the range $\Omega:=h(\mathbb{D})$ of the conformal map $h$, then the $\lambda$- spiral connecting $z$ and $0$ is a subset of $\Omega$.
	\[e^{-\lambda t}z\in\Omega,\qquad t\geq 0,\qquad z\in\Omega.\]

\end{enumerate}

If $\lambda\in(0,\infty)$ then the map $h$ is called starlike.

In the non-elliptic case, $h$ is starlike at infinity with respect to $\tau\in\mathbb{T}$:
\begin{equation}\label{koenignel}
\Im \left(\overline{\tau}(\tau-z)^2h'(z)\right)\geq 0,
\end{equation}
or equivalently
\begin{enumerate}[(i)]
	\item The function $h$ is conformal in the unit disk and $\limsup\limits_{z\rightarrow \tau}\Im h(z)=+\infty$.
	\item If a point $z$ exists in the range $\Omega:=h(\mathbb{D})$ of the conformal map $h$, then for every $t\geq 0$
	\[z+it\in\Omega.\]
	
\end{enumerate}
Our starting point was the observation that Brennan's conjecture holds for the Koening map of a semigroup. To prove this result, which is well known to the experts in this field, first we need to recall some classical theorems.
\begin{Littlewood}[\cite{SHAP93}]
Let $f$ be a holomorphic and $g$ be a conformal function in the unit disk with $f(\mathbb{D})\subset g(\mathbb{D})$ and $f(0)=g(0)$. Then for every $p>0$ and $r\in(0,1)$
\[\int\limits_0^{2\pi}|f(re^{i\theta})|^p\,d\theta\leq \int\limits_0^{2\pi}|g(re^{i\theta})|^p\,d\theta.\]
\end{Littlewood}

\begin{Koebe}[\cite{POM75}]
Let $f$ be a conformal map in the unit disk with $f(0)=0$ and $f'(0)=1$. Then
\begin{enumerate}[(i)]
\item \[\frac{|z|}{(1+|z|)^2}\leq |f(z)|\leq \frac{|z|}{(1-|z|)^2}.\]
\item \[\frac{1-|z|}{(1+|z|)^3}\leq |f'(z)|\leq \frac{1+|z|}{(1-|z|)^3}.\]
\item \[\frac{1-|z|}{1+|z|}\leq \frac{|zf'(z)|}{|f(z)|}\leq \frac{1+|z|}{1-|z|}.\]
\end{enumerate}
\end{Koebe}
A direct consequence of the Hardy-Stein identity is the following result for the integral means of a conformal map:
\begin{theorem}[\cite{POM92}]\label{integralmeans}
If $f$ is a conformal function in the unit disk and $f(z)=O\left((1-|z|)^{-a}\right),\, a\in (0,2]$ as $|z|\rightarrow 1^-$. Then for every $p>\frac{1}{a}$
\[\int\limits_0^{2\pi}|f(re^{i\theta})|^p\,d\theta= O\left(\frac{1}{(1-r)^{ap-1}}\right),\qquad r\rightarrow 1^-.\]
For $0<p<\frac{1}{a}$ the associated $p$-integral means are uniformly bounded.
\end{theorem}
A well-known result of this theory is that the conjecture holds for the classes of close-to-convex and starlike functions. For these results, we refer to the original paper by J. E. Brennan \cite{BRE78} in which they are attributed to B. Dahlberg and J. Lewis
\begin{proposition}\label{BrennanKoenig}
Let $\{\phi_t(z)\}_{t\geq 0}$ be a non-trivial continuous semigroup of conformal self-maps of the unit disk and $h$ be the associated Koenigs map. Then, for every $p\in(-2,\frac{2}{3})$
\begin{equation}
\int\limits_{\mathbb{D}}\left|h'(w)\right|^p\,dA(w)<\infty.
\end{equation}
\end{proposition}
\begin{proof}
Without loss of generality, we can assume that $p\in(-2,-1)$. By \eqref{koenigel}, \eqref{koenignel} and the Koebe distortion theorem there exist a holomorphic function $f$ with $\Re f\geq 0$ such that
\[|h'(z)|^p\lesssim  |f(z)|^{|p|},\qquad 0<\delta<|z|<1.\]
We consider the function $g(z)=\frac{\overline{f(0)}z+f(0)}{1-z}$ which maps the unit disk onto the right half-plane. We apply Littlewood's subordination principle and Theorem~\ref{integralmeans}, yielding to
\begin{align*}
\int\limits_{\mathbb{D}}\left|h'(w)\right|^p\,dA(w)&\lesssim  \int\limits_{\delta}^1\int\limits_{0}^{2\pi}\left|h'(re^{i\theta})\right|^p\,d\theta\,dr\\
&\lesssim \int\limits_{\delta}^1\int\limits_{0}^{2\pi}\left|g(re^{i\theta})\right|^{|p|}\,d\theta\,dr\\
&\lesssim \int\limits_{\delta}^1\frac{1}{(1-r)^{|p|-1}}\,dr<\infty.\qedhere
\end{align*}
\end{proof}

\subsection{B\'{e}koll\'{e}-Bonami weights and integration operators}
Given a non-negative, integrable function $\omega$ on the unit disk, we denote by $L^p(\omega),\,p>0$ the $L^p$ space with respect to the measure $\omega(z)\,dA(z)$. Let $L_\alpha^p(\omega)$ be the corresponding Bergman space of analytic functions, \cite{HKZ00}.

B\'{e}koll\'{e}-Bonami weights are the analogs of the well-known Muckenhoupt weights, \cite{CR85}.
For $p>1,\,\eta>-1$ the class of weights $B_p(\eta)$ consists of all weights $\omega$ such that
\begin{equation}\label{bekk}
\int\limits_{S(\theta,h)}\omega(z)\,dA_\eta(z)\left(\int\limits_{S(\theta,h)}\omega^\frac{-1}{p-1}(z)\,dA_\eta(z)\right)^{p-1}\lesssim  \left(A_\eta\left(S(\theta,h)\right)\right)^p,
\end{equation}
where $dA_\eta(z)=(1+\eta)\left(1-|z|^2\right)^\eta\,dA(z)$ and $S(\theta,h)$ is an arbitrary Carleson box
\begin{equation*}
S(\theta,h)=\left\{z=re^{it}: 1-h<r<1,\,\left|t-\theta\right|<\frac{h}{2} \right\}.
\end{equation*}
The implied constant in \eqref{bekk} is independent of the Carleson box.
If $\eta=0$ we will write $B_p$ instead of $B_p(0)$.

In \cite{BE81}, B\'{e}koll\'{e} proved that the Bergman projection 
\begin{equation*}
P_\eta(f)(z)=\int\limits_{\mathbb{D}}\frac{f(w)}{\left(1-\overline{z}w\right)^{\eta+2}}\,dA_\eta(z)\qquad \eta>-1,
\end{equation*}
is bounded from $L^p(\omega),\,p>1$ onto $L_\alpha^p(\omega)$ if and only if $\frac{\omega(z)}{\left(1-|z|^2\right)^\eta}\in B_p(\eta)$.

We list a couple of lemmata, essential for our purpose.
\begin{theorem}[\cite{AC09}]\label{AlCO}
Let $p_0>1,\,\eta>-1$ and $\omega(z)$ be a weight such that 
\[\frac{\omega(z)}{\left(1-|z|^2\right)^\eta}\in B_{p_0}(\eta).\]
Then, for every $p>0$ and every $f\in L_\alpha^p(\omega)$
\begin{equation}\label{LP2}
\int\limits_{\mathbb{D}} \left|f(z)\right|^p\omega(z)\,dA(z)\sim \left|f(0)\right|^p+\int\limits_{\mathbb{D}} \left|f'(z)\right|^p\left(1-|z|^2\right)^p\omega(z)\,dA(z).
\end{equation}
\end{theorem}
\begin{theorem}[\cite{CON10}]\label{CO}
Let $\mu$ be a positive Borel measure in the unit disk and $\omega(z)$ be a weight such that $\frac{\omega(z)}{\left(1-|z|^2\right)^\eta}\in B_{p_0}(\eta),\,p_0>1,\,\eta>-1$. Then, the property of $\mu$ being a Carleson measure for $L_\alpha^p(\omega)$ is independent of $p>0$.
\end{theorem}

A holomorphic self-map of the unit disk $\phi$ induces a (formal) composition operator $C_\phi(f)= f\circ \phi$. A direct consequence of the theorem above is that for B\'{e}koll\'{e}-Bonami weights $\omega(z)$, the property of $C_\phi$ to be bounded on $L_\alpha^p(\omega)$ is independent of $p>0$.
\subsection{Integration operators}
This is a well-studied topic that attracted the interest of plenty of mathematicians the last 30 years, see for example \cite{AL07,AC01, Ale_Arist,nikos24, contpel,GirPel06,Pel23,Rt07,Volberg17}.
Every function $g$ that belongs to the Bloch spaces corresponds to a bounded integration operator $T_g$ on $L_\alpha^p,\,p>0$:
\[T_g(f)=\int_0^zf(w)g'(w)\,dA(w).\]
Recall that the Bloch spaces $\mathfrak{B}$ contains all analytic functions $g:\mathbb{D}\mapsto \mathbb{C}$ that are Lipschitz continuous when the disk and the plane are equipped with the hyperbolic distance, see \cite{BM07}, and the euclidean metric, respectively. Or equivalently $\mathfrak{B}$ consists of all analytic functions $g$ in the unit disk such that
\[\sup_\mathbb{D}\{(1-|z|^2)|g'(z)|\}<\infty.\]
The spectrum of the operator $T_g$ has been determined in terms of B\'{e}koll\'{e}-Bonami weights in \cite{AC09,APR19}.
\begin{theorem}\label{tg}
 The point $\lambda\in\mathbb{C}\setminus\{0\}$ belongs to the resolvent set of $T_g$ acting on $L_\alpha^p,\, p\geq 1$ if and only if 
 \[e^{p\Re\frac{g}{\lambda}} \in B_\infty.\]
 If we further assume that the derivative of symbol $g$ is the Cauchy transform of a finite measure $\mu$ on the unit circle, that is:
\begin{equation}\label{cauchy}
g'(z)=\int_\mathbb{T}\frac{1}{\zeta-z}\,d\mu(\zeta),
\end{equation}
 then 
 \[\sigma\left(T_g\lvert_{L^p_\alpha}\right)=\{0\}\cup\overline{\left\{\lambda:e^{\frac{\Re g}{\lambda}}\notin L^p_\alpha\right\}}.\]
 
\end{theorem}
The class $B_\infty$ contains all weights $\omega$ in $\mathbb{D}$ such that
\[\sup\frac{1}{\int_{S(\theta,h)}\omega}\int_{S(\theta,h)}M\left(\omega\mathbf{1}_{S(\theta,h)}\right)<\infty,\]
where $M$ is the Hardy-Littlewood maximal function (over Carleson boxes) and the suppremum is taken over all Carleson boxes.

We extract the following lemma from \cite{APR19}:

\begin{lemma}\label{tglem}
 Let $g\in \mathfrak{B}$ and $r\in\mathbb{R}$. Then, the weight $\omega:=e^{r\Re g}$ belongs in $B_\infty$ if and only if there exists $q>1$ such that $\omega\in B_q$ if and only if there exists a $\gamma>0$ such that
 \begin{equation}
     \int_{\mathbb{D}}\frac{\omega(w)}{|1-\overline{z}w|^{\gamma+2}}\,dA(w)\lesssim  \frac{\omega(z)}{\left(1-|z|^2\right)^\gamma},\qquad z\in\mathbb{D}.
 \end{equation}
\end{lemma}

\section{Proof of the main result}
\begin{lemma}\label{bek}
The weight $\omega(z)=\left|h(z)\right|^{p},\,p\in(-2,-1)$, where $h$ is the Koenings map of a holomophic semigroup, exists in $B_q$ for some $q>1$.
\end{lemma}
\begin{proof}
Let us first assume that $h$ is starlike at infinity. For every automorphism of the unit disk $T_w(z)=\frac{w-z}{1-\overline{w}z},\,w\in\mathbb{D}$, there exists a holomorphic function $u:\mathbb{D}\rightarrow\{z:\Re z\geq0\}$ such that
\begin{equation*}
\left(h\circ T_w\right)'(z)=c_1\frac{u(z)}{(c_2-z)^2},\qquad c_1,\,c_2\in \mathbb{T},
\end{equation*}
one can prove that $c_1=iT_w(\tau)$ and $c_2=T_w(\tau)$, where $\tau$ is the Denjoy-Wolff point of the semigroup. By Littlewood's subordination principle applied as in Proposition~\ref{BrennanKoenig}, we have that
\begin{align*}
\int\limits_{\mathbb{D}} \left|\left(h\circ T_w\right)'(z)\right|^p\,dA(z)&\lesssim \int\limits_{\mathbb{D}}  \left|\frac{1}{u(z)}\right|^{|p|}\,dA(z)=\int\limits_{\mathbb{D}}  \left|\frac{\alpha-\overline{\alpha}z}{1-z}\right|^{|p|}\,dA(z)\\
&\leq C(p)|\alpha|^{|p|},
\end{align*}
where $|\alpha|=\left| \frac{1}{u(0)}\right|=  \frac{1}{\left|h'(w)\right|(1-|w|^2)}$.
Taking into account the behaviour of the automorphism $T_w$ in the Carleson box $S(\theta,h)$ with center $w=(1-\frac{h}{2})e^{i\theta}\in\mathbb{D}$, we obtain:
\begin{align*}
\frac{1}{\left|S(\theta,h)\right|}\int\limits_{S(\theta,h)}\left|h'(z)\right|^p\,dA(z)&\lesssim \left(1-|w|^2\right)^{|p|}\int\limits_{\mathbb{D}} \left|\left(h\circ T_w\right)'(z)\right|^p\,dA(z)\\
&\leq C(p)\left|h'(w)\right|^{p}.
\end{align*}
Working similarly or applying directly the following result of Feng and MacGregor \cite{FMCG76,HH2002}: For every conformal function $f$ in the unit disk and for every $b\in[0,\frac{2}{3})$
\[\int\limits_{\mathbb{D}}|f'(z)|^b\,dA(z)\lesssim  |f'(0)|^{b},\]
one can prove that
\begin{align*}
\frac{1}{\left|S(\theta,h)\right|}\int\limits_{S(\theta,h)}\left|h'(z)\right|^\frac{|p|}{3}\,dA(z)&\lesssim \left(1-|w|^2\right)^{\frac{p}{3}}\int\limits_{\mathbb{D}} \left|\left(h\circ T_w\right)'(z)\right|^{\frac{|p|}{3}}\,dA(z)\\
&\leq C(p)\left|h'(w)\right|^{\frac{|p|}{3}}.
\end{align*}
Thus, $|h'|^p\in B_4$.

For the spirallike case we can assume without loss of generality  that $h(0)=0$. By the Koebe distortion theorem the function $g(z)=\log\frac{h(z)}{z}$ belongs in the Bloch space. Furthermore, since $h$ is spirallike
\[h'(z)=\lambda \frac{h(z)}{z}u(z),\qquad \Re u\geq 0.\]
Herglotz representation formula \cite{POM92} yields to the existence of a finite measure $\mu$ on the unit circle such that
\[u(z)=\int_\mathbb{T}\frac{\zeta+z}{\zeta-z}d\mu(\zeta),\qquad z\in\mathbb{D},\]
and $\mu\left(\mathbb{D}\right)=\frac{1}{\lambda}$. Consequently, the derivative of the symbol $g$
is the Cauchy transform of a finite measure in the unit circle:
\[g'(z)=\frac{\lambda}{z}\left(u(z)-\frac{1}{\lambda}\right)=2\lambda\int_\mathbb{T}\frac{d\mu(\zeta)}{\zeta-z}.\]
The integration operator $T_g$ is bounded on $L^2_a$ and Theorem~\ref{tg} imply that $2\frac{r-p}{rp},$ where $r\in(-2,p)$, exists in the resolvent set of $T_g\lvert_{L^2_\alpha}$.

Therefore, by Lemma~\ref{tglem} $$\omega_1:=\left|\frac{h(z}{z}\right|^{\frac{pr}{r-p}}\in B_\infty,$$ and there exists a $\gamma>0$ such that

 \begin{equation}\label{part1}
     \int_{\mathbb{D}}\frac{\omega_1(w)}{|1-\overline{z}w|^{\gamma+2}}\,dA(w)\lesssim  \frac{\omega_1(z)}{\left(1-|z|^2\right)^\gamma},\qquad z\in\mathbb{D}.
 \end{equation}
Applying Littlewood's subordination principle in a similar manner us before, we obtain that 
 \begin{equation}\label{part2}
     \int_{\mathbb{D}}\frac{|u(w)|^r}{|1-\overline{z}w|^{4}}\,dA(w)\lesssim  \frac{|u
     (z)|^{r}}{\left(1-|z|^2\right)^2},\qquad z\in\mathbb{D}.
 \end{equation}
Let $\delta=2\frac{p}{r}+\gamma\frac{r-p}{r}>0$, Hölder's inequality yields to 
\begin{multline*}
     \int_{\mathbb{D}}\frac{\omega(w)(1-|z|^2)^{\delta}}{|1-\overline{z}w|^{\delta+2}}\,dA(w)\\\lesssim\left(\int_{\mathbb{D}}\frac{|u(w)|^r\left(1-|z|^2\right)^2}{|1-\overline{z}w|^{4}}\,dA(w)\right)^{\frac{p}{r}}\left(\int_{\mathbb{D}}\frac{\omega_1(w)\left(1-|z|^2\right)^\gamma}{|1-\overline{z}w|^{\gamma+2}}\,dA(w)\right)^{\frac{r-p}{r}} \\
  \lesssim \omega(z).
\end{multline*}

The proof follows applying again Lemma~\ref{tglem}, note that $\log h'\in\mathfrak{B}$.

\end{proof}
The two cases in the above lemma can be merged. The argument that we used when the Koenigs map is spirallike is far more general. Although, we decided to keep the first elementary argument for expository reasons.

The last ingredient that we need for the proof of Theorem~\ref{main} is an inequality similar to Schwarz-type lemmata for conformal functions. We will rely on a well-know technique of Sergei 
Shimorin, see for example \cite{SHI05} where he proved using the Löwner differential equation that every conformal self-map $\phi$ of the unit disk with $\phi(0)=0$, satisfies the following inequality:
\begin{equation}\label{shimorin}
\frac{1-|z|^2}{1-|\phi(z)|^2}\leq \left|\phi'(w)\right|\left(\frac{|z|}{|\phi(z)|}\right)^2.
\end{equation}

\begin{lemma}\label{ShimLow}
	Let $\phi:\mathbb{D}\mapsto\mathbb{D}$ be a conformal map that fixes the origin. Then, for every $p\in(0,2)$.
	\begin{equation}\label{lowner}
	\left|\phi'(z)\right|^p\geq \left(\frac{1-|z|^2}{1-|\phi(z)|^2}\right)^2\left|\frac{\phi(z)}{z}\right|^{2p}.
	\end{equation}
\end{lemma}
\begin{proof}
By a standard density argument, we can assume that $\phi$ is a slit mapping, meaning that $\phi(\mathbb{D})=\mathbb{D}\setminus\widehat{\gamma}$, where $\gamma$ is a smooth Jordan arc in $\mathbb{D}\setminus\{0\}$ with starting point in the unit circle.
There exists a Löwner chain (\cite{GOL69}) $\{\varphi_t\}_{0\leq t\leq T}$ such that
\begin{enumerate}[(i)]
\item $$\varphi_0(z)=z,\qquad \varphi_T(z)=\phi(z),\qquad z\in\mathbb{D}.$$
\item $$\frac{\partial}{\partial_t}\varphi_t(z)=-\varphi_t(z)\frac{1+k(t)\varphi_t(z)}{1-k(t)\varphi_t(z)},$$
where $k(t)$ is continuous with $|k(t)|=1$.
\end{enumerate}
We observe that:
\begin{align*}
\frac{\partial}{\partial_t}\log\frac{\left|\varphi'_t(z)\right|^{-p}\left|\frac{\varphi_t(z)}{z}\right|^{2p}}{\left(1-\left|\varphi_t(z)\right|^2\right)^2}&=-p\frac{\Re\left(\overline{\varphi'_t(z)}\frac{\partial}{\partial_t}\varphi'_t(z)\right)}{\left|\varphi'_t(z)\right|^2}\\
&+2p\frac{\Re\left(\overline{\varphi_t(z)}\frac{\partial}{\partial_t}\varphi_t(z)\right)}{\left|\varphi_t(z)\right|^2}+4\frac{\Re\left(\overline{\varphi_t(z)}\frac{\partial}{\partial_t}\varphi_t(z)\right)}{1-\left|\varphi_t(z)\right|^2}\\
&=\left(-p-4\frac{|\varphi_t(z)|^2}{1-|\varphi_t(z)|^2}\right)\Re \left(\frac{1+k(t)\varphi_t(z)}{1-k(t)\varphi_t(z)}\right)\\
&+p\Re\left(\frac{2k(t)\varphi_t(z)}{(1-k(t)\varphi_t(z))^2}\right)\\
&=-\left|1-k(t)\varphi_t(z)\right|^{-2}\left((4-p)\left|\varphi_t(z)\right|^2+p\right)\\
&+p\Re\left(\frac{2k(t)\varphi_t(z)}{(1-k(t)\varphi_t(z))^2}\right)\\
&\leq-\frac{(4-p)\left|\varphi_t(z)\right|^2-2p\left|\varphi_t(z)\right|+p}{\left|1-k(t)\varphi_t(z)\right|^{2}}\\
&\leq 0.
\end{align*}
Thus
\begin{equation*}
\frac{\left|\varphi'_t(z)\right|^{-p}\left|\frac{\varphi_t(z)}{z}\right|^{2p}}{\left(1-\left|\varphi_t(z)\right|^2\right)^2}\leq \frac{\left|\varphi'_0(z)\right|^{-p}\left|\frac{\varphi_0(z)}{z}\right|^{2p}}{\left(1-\left|\varphi_0(z)\right|^2\right)^2}=\left(1-|z|^2\right)^{-2}.
\end{equation*}
\end{proof}
Applying the above lemma in conjunction with the Koebe distortion theorem, we obtain:
\begin{lemma}\label{LOW}
For every conformal map $\phi:\mathbb{D}\mapsto\mathbb{D}$ and every $p\in(0,2)$, there exists a constant $C>0$ such that
\begin{equation}\label{brennaneq}
\left|\phi'(w)\right|^{\frac{p}{2}}\geq C \frac{1-|z|^2}{1-|\phi(z)|^2},\qquad z\in \mathbb{D}.
\end{equation}
\end{lemma}
\begin{proof}
Let us first assume that $\phi(0)=0$. Then, by Lemma~\ref{ShimLow} and Koebe's distortion theorem
\[\left|\phi'(z)\right|^{\frac{p}{2}}\geq \frac{1-|z|^2}{1-|\phi(z)|^2}\left|\frac{\phi(z)}{z}\right|^p\geq \left|\frac{|\phi'(0)|}{4}\right|^p\frac{1-|z|^2}{1-|\phi(z)|^2}.\]
If $\alpha=\phi(0)\neq0$ we consider the map $T_\alpha\circ\phi(z)=\frac{\alpha-\phi(z)}{1-\overline{\alpha}\phi(z)}$. By the previous case
 \[\left|\phi'(z)\right|^{\frac{p}{2}}\left|T_\alpha'\circ\phi(z)\right|^{\frac{p}{2}}\geq \left|\frac{|\phi'(0)|}{4(1-|\phi(0)|^2)}\right|^p\frac{1-|z|^2}{1-|\phi(z)|^2}\left|T_\alpha'\circ\phi(z)\right|^{-1}.\]
This completes the proof since $\left|T_\alpha'\circ\phi(z)\right|$ is bounded from above.
\end{proof}
\begin{proof}[\textbf{Proof of Theorem \ref{main}}]
We can assume that $\left\{\phi_t\right\}_{t\geq 0}$ is not a group and that $p\in(-2,-1)$.
\begin{align*}
\int\limits_{\mathbb{D}}\left|\phi_t'(w)\right|^p\,dA(w)&= C(t,p)\int\limits_{\mathbb{D}}\left|h'\circ\phi_t(w)\right|^{|p|}\left|h'(w)\right|^{p}\,dA(w)\\
&=C(t,p)\norm{C_{\phi_t}(h')}^{|p|}_{L^{|p|}(\omega)},
\end{align*}
where $\omega(z)=\left|h'(z)\right|^p$. By Lemma~\ref{bek} and Theorem~\ref{CO}, it is sufficient to prove that the composition operator $C_{\phi_t}$ is bounded on $L_\alpha^2(\omega)$. 

For an arbitrary function $f\in L_\alpha^2(\omega)$ we apply the Littlewood-Paley formula~\eqref{LP2}, yielding to:
\begin{align*}
&\norm{C_{\phi_t}(f)}_{L^2(\omega)}^2=\int\limits_{\mathbb{D}}\left|f\left(\phi_t\right)(w)\right|^2\left|h'(w)\right|^p\,dA(w)\\
&\lesssim  \left|f(\phi_t(0))\right|^2+\int\limits_{\mathbb{D}}\left|f'\left(\phi_t\right)(w)\right|^2\left|\phi_t'(w)\right|^2\left(1-|w|^2\right)^2\left|h'(w)\right|^p\,dA(w)\\
&\lesssim  \left|f(\phi_t(0))\right|^2+\int\limits_{\phi_t\left(\mathbb{D}\right)}\left|f'(w)\right|^2\left(1-|\phi_t^{-1}(w)|^2\right)^2\left|h'(\phi_t^{-1}(w))\right|^p\,dA(w).
\end{align*}
The point evaluation linear functionals are bounded since the weight $\omega$ is B\'{e}koll\'{e}-Bonami. Thus, applying again the Littlewood-Paley formula~\eqref{LP2} it is easy to see that the following inequality suffices boundedness of $C_{\phi_t}$ on $L_\alpha^2(\omega)$: 

\begin{equation*}
\left(1-|z|^2\right)^2\left|h'(z))\right|^p\lesssim \left(1-|\phi_t(z)|^2\right)^2\left|h'(\phi_t(z))\right|^p,\qquad z\in \mathbb{D},
\end{equation*}
or equivalently that
	\begin{equation}
\left|\phi_t'(z)\right|^{|p|}\geq C(\phi_t) \left(\frac{1-|z|^2}{1-|\phi_t(z)|^2}\right)^2.
\end{equation}
The proof follows from Lemma~\ref{LOW}
\end{proof}

	\bibliographystyle{amsplain-nodash} 
	\bibliography{refbre} 
\end{document}